\documentclass{amsart}

\usepackage{amssymb}
\usepackage{mathrsfs}
\usepackage{xypic}
\usepackage{url}
\usepackage[dvipsnames]{xcolor}
\usepackage{graphicx}
\usepackage{etoolbox}
\usepackage{hyperref}
\usepackage{bbding}
\usepackage{pifont}
\usepackage{tikz}

\newtheorem{theorem}{Theorem}[section]
\newtheorem{lemma}[theorem]{Lemma}
\newtheorem{proposition}[theorem]{Proposition}

\theoremstyle{definition}
\newtheorem{definition}[theorem]{Definition}
\newtheorem{conjecture}[theorem]{Conjecture}
\newtheorem{remark}[theorem]{Remark}
\theoremstyle{plain}
\numberwithin{equation}{section}
\numberwithin{figure}{section}
\numberwithin{table}{section}

\begin{document}

\title{A Weak Solution of Inverse Treblecross}

\author{Kai Liang}
\address{School of Mathematics and Statistics\\Xidian University\\266 Xinglong Section, Xifeng Road\\710126 Xi'an, Shaanxi\\China}
\email{24071213197@stu.xidian.edu.cn}

\author{Muxi Li}
\address{Department of Mathematics\\School for Mathematics Science\\Tiangong University\\Binshuixi Road 399\\300190 Tianjin\\China}
\email{limuxi@ustc.edu.cn}

\begin{abstract}
We study \emph{Inverse Treblecross}, the impartial avoidance version of the game Treblecross, in which
completing a line of three X's is prohibited. We give a weak solution of the game: an empty board of
$n$ cells is a $\mathscr{P}$-position (Sprague--Grundy value $0$) if and only if
$n\equiv 0,2,3,6,9\pmod{10}$, and has value $1$ if $n\equiv 1,4,5,7,8\pmod{10}$. In particular the
empty-board Sprague--Grundy sequence is the purely periodic sequence of period $10$ with values in
$\{0,1\}$. The proof introduces two characteristic functions on $\mathbb Z/10\mathbb Z$ and a family of
\emph{regular positions}, and shows by induction that such a position splits, under a reduction move,
into a disjunctive sum of regular positions whose sign product is preserved or inverted.
\end{abstract}



\maketitle

\definecolor{green}{rgb}{0.0, 0.5, 0.05}
\definecolor{teal}{rgb}{0.0, 0.5, 0.5}
\definecolor{purple}{rgb}{0.5, 0.0, 0.5}
\definecolor{fuchsia}{rgb}{1.0, 0.0, 1.0}
% see latexcolor.com;  not all names are defined it seems, hence the definition for \fuchsia here.

\long\def\blue#1{{\color{blue}#1}}
\long\def\green#1{{\color{green}#1}}
\long\def\red#1{{\color{red}#1}}

\long\def\teal#1{{\color{teal}#1}}
\long\def\purple#1{{\color{purple}#1}}

\def\nocolors{\long\def\blue##1{##1}\long\def\green##1{##1}\long\def\red##1{##1}}

% \nocolors

\def\nocoloredcomments{\long\def\blue##1{}\long\def\green##1{}\long\def\red##1{}}

% \nocoloredcomments

% A useful definition for boxing things
\def\boxit#1{\lower4pt\hbox{\vrule\vbox{\hrule\vskip5pt\hbox{\vbox{\hsize=10cm\noindent#1}}\vskip3pt\hrule}\vrule}}

% put in a warning about preliminary version
\def\prelimwarning{\centerline{\boxit{\centerline{EXTREMELY PRELIMINARY VERSION}\\\centerline{DO NOT DISTRIBUTE}}}\medskip}

% a useful box after it has been submitted to remember which version it was
% \def\prelimwarning{\centerline{\boxit{\centerline{SUBMITTED VERSION}}}\medskip}

% \prelimwarning

% create a timestamp at the beginning of the paper, and at top and bottom of pages 2, etc.

\def\boxit#1{\lower4pt\hbox{\vrule\vbox{\hrule\vskip3pt\hbox{\vbox{\hsize=10cm\noindent#1}}\vskip3pt\hrule}\vrule}}
\def\timestring{\count0=\time \divide\count0 by 60 \count2=\count0 \multiply\count0 by 60 \count4=\time \advance\count4 by -\count0%
\count6=10 
\ifnum\count4 <\count6 \toks0={0} \else \toks0={}\fi%
\number\count2 :\the\toks0 \number\count4}
\def\datetimewide{Date:\ \today\qquad\qquad Time:\ \timestring}
\def\timestamp{\centerline{\boxit{\hfil\datetimewide\hfil}}\medskip}
\def\datetimeshort{\today\quad{\timestring}}

%\timestamp

% The timestamp headers and the to-do list have been removed; uncommenting the next lines
% restores the (submission-time) page stamps if ever needed.
%\makeatletter
%\def\@oddhead{\tiny page \thepage\hfil\datetimeshort\hfil page \thepage}\def\@evenhead{\tiny page \thepage\hfil\datetimeshort\hfil page \thepage}
%\def\@oddfoot{\tiny page \thepage\hfil page \thepage}\def\@evenfoot{\tiny page \thepage\hfil page \thepage}
%\makeatother

% In order to create double space uncomment the following line:
% \baselineskip=2\baselineskip

\newcommand{\varG}{\mathscr{G}}

\newcommand{\E}{
    \tikz[baseline=-0.6ex]{
        \draw[thick] (0,0) circle (1ex);
    }
}

\newcommand{\X}{
    \tikz[baseline=-0.6ex]{
        \draw[thick, line cap=round] (-0.9ex, -0.9ex) -- (0.9ex, 0.9ex);
        \draw[thick, line cap=round] (-0.9ex, 0.9ex) -- (0.9ex, -0.9ex);
    }
}
\newcommand{\MT}[1]{\overline{#1}}
\newcommand{\ResT}{\mathbb{Z}/10\mathbb{Z}}
\newcommand{\D}{\diamond_3}

\newcommand{\bbD}{\mathbb{D}}
\newcommand{\mex}{\mathrm{mex}}
\newcommand{\nex}{\mathrm{next}}
\section{Introduction} \label{sec:intro}
Combinatorial games of ``making and avoiding'' constitute a well-studied and richly varied domain
within combinatorial game theory. While the majority of classic positional games in this
category---such as Tic-Tac-Toe, Gomoku, and Connect-4---are partisan, featuring players with distinct
sets of pieces, there is profound mathematical interest in their impartial counterparts, where both
players share a single reservoir of indistinguishable pieces.

A paradigmatic example of such an impartial game is ``Treblecross'', analyzed by Berlekamp, Conway,
and Guy in \textit{Winning Ways} \cite{BCG01}. Played on a $1\times n$ strip of cells, both players
take turns placing an ``X'', and the first player to complete three consecutive X's wins. In the
framework of combinatorial game theory, Treblecross is the octal game $0.007$ (a move removes a gap of
length at least three and leaves at most one cell on each side). Despite its apparent simplicity and
decades of study, it remains open whether the Sprague--Grundy sequence of Treblecross is ultimately
periodic.

In this paper we study the avoidance version of Treblecross, which we call \emph{Inverse Treblecross}
(Beck \cite{Bec08} calls this kind of avoidance game a ``reverse'' game). In \emph{Inverse
Treblecross}, completing a line of three is
prohibited: a move is legal only if it does not create three consecutive X's. Under the normal play
convention a player loses when they face a position with no legal move.

It is instructive to compare Inverse Treblecross with its lower-order analogue, ``Inverse Duocross''
(avoiding two-in-a-row). In Inverse Duocross a move renders the played cell and its two neighbours
unplayable, so a single move cleanly severs the board into independent empty segments; the game is
structurally equivalent to the normal play of Dawson's Chess, which decomposes into disjunctive sums
and has a famously periodic SG sequence.

The difficulty in analyzing Inverse Treblecross is that it resists such a decomposition: because two
X's may be adjacent, the legality of a move in an empty segment depends on the configuration of the
pieces bounding that segment, and an arbitrary move is not simply a nim-sum of independent
subcomponents. This forces the more nuanced state-space analysis developed in Section~\ref{sec:main}.

Our objective is to evaluate the game from its initial (empty) state; in the standard terminology this
is a \emph{weak solution}. Evaluating the exact SG value of an arbitrary partially filled board (a
\emph{strong solution}) remains open. We resolve the empty-board SG sequence completely and show that
its values are confined to $\{0,1\}$.

\begin{theorem}\label{thm:intro}
An empty board of $n$ cells has Sprague--Grundy value $0$ (the second player wins) if and only if
$n\equiv 0,2,3,6,9\pmod{10}$, and value $1$ if $n\equiv 1,4,5,7,8\pmod{10}$.
\end{theorem}

\section{Preliminaries}\label{sec:prelim}

\paragraph{Background.} We assume the reader is familiar with the basics of impartial combinatorial
game theory under normal play: the Sprague--Grundy (SG) function, its value of a disjunctive sum as the
nim-sum ($\oplus$, the bitwise XOR) of the values of the summands, and the classification of positions
into $\mathscr{P}$-positions (SG value $0$, the Next Player loses) and $\mathscr{N}$-positions (SG value
$\neq 0$, the Next Player wins). A game position with no legal move is \emph{terminal} and has value
$0$. A standard reference is Winning Ways \cite{BCG01}.

\paragraph{Board notation.} A board is a $1\times n$ strip of cells. A position is written as a string
in the symbols below, and square brackets $[\ ]$ mark the two ends of the board.
\begin{itemize}
    \item \emph{Pieces:} the symbol $X$ denotes an occupied cell.
    \item \emph{Gaps:} a natural number $n$ denotes a run of $n$ consecutive empty cells.
    \item \emph{Residue classes:} the symbol $\bar n$ denotes a run of empty cells of length
    $k\equiv n\pmod{10}$; a string containing $\bar n$ thus denotes the whole set of positions with that
    congruence condition.
\end{itemize}
We write $[w]\mapsto[v]$ to mean that the position $[v]$ is reachable from $[w]$ by one legal move, and
$[w]\cong[v]$ for game equality (equal SG value, in fact equality as games). A disjunctive sum is
written with $\oplus$.

\paragraph{Characteristic functions.} Since $\gcd(3,10)=1$, the element $3$ is invertible in
$\ResT$ with inverse $7$; the map $x\mapsto x/3$ is therefore the same as $x\mapsto 7x$. We use $7x$
(rather than $x/3$) throughout, and define two characteristic functions on $\ResT$:
\begin{equation}\label{eq:chi}
\chi_1(a)=\begin{cases}
    1,  &7a\in\{\bar0,\bar1,\bar2,\bar3,\bar4\},\\
    -1, &\text{otherwise},
\end{cases}
\qquad
\chi_2(a)=\begin{cases}
    1, &7a\in\{\bar9,\bar0,\bar1\},\\
    -1,&7a\in\{\bar4,\bar5,\bar6\},\\
    0, &\text{otherwise}.
\end{cases}
\end{equation}
Their values are listed in Table~\ref{tab:chi}.
\begin{table}[htbp]
  \centering
  \caption{Values of $\chi_1$ and $\chi_2$ on $\ResT$.}
  \label{tab:chi}
  \begin{tabular}{c|cccccccccc}
    $a$ & $0$ & $1$ & $2$ & $3$ & $4$ & $5$ & $6$ & $7$ & $8$ & $9$ \\ \hline
    $\chi_1(a)$ & $1$ & $-1$ & $1$ & $1$ & $-1$ & $-1$ & $1$ & $-1$ & $-1$ & $1$ \\
    $\chi_2(a)$ & $1$ & $0$ & $-1$ & $1$ & $0$ & $-1$ & $0$ & $1$ & $-1$ & $0$ \\
  \end{tabular}
\end{table}

\begin{figure}[htbp]
    \centering
    \begin{tikzpicture}[scale=0.7]
        \foreach \ang/\v in {-18/0, 2*36-18/2, 3*36-18/3, 6*36-18/6, 9*36-18/9} {
            \filldraw[draw=black,fill=white] 
                (\ang:1.5) -- (\ang:2.5) 
                arc (\ang:\ang+36:2.5) -- (\ang+36:1.5) 
                arc (\ang+36:\ang:1.5) -- cycle;
            \node at (\ang+18:3) {\Large \v};
        }
        \foreach \ang/\v in {1*36-18/1, 4*36-18/4, 5*36-18/5, 7*36-18/7, 8*36-18/8} {
            \filldraw[draw=black,fill=gray!50] 
                (\ang:1.5) -- (\ang:2.5) 
                arc (\ang:\ang+36:2.5) -- (\ang+36:1.5) 
                arc (\ang+36:\ang:1.5) -- cycle;
            \node at (\ang+18:3) {\Large \v};
        }
        \foreach \ang in {0, 3*36, 7*36} {
            \filldraw[draw=black,fill=white] (\ang:2) circle (0.25);
        }
        \foreach \ang in {2*36, 5*36, 8*36} {
            \filldraw[draw=black,fill=gray] (\ang:2) circle (0.25);
        }
        \foreach \ang in {1*36, 4*36, 6*36, 9*36} {
            \node at (\ang:2) {\large $\times$};
        }
        \filldraw[draw=black,fill=white] (3.9,1.9)--(3.9,2.6)--(4.6,2.6)--(4.6,1.9)--cycle; 
        \node at (4.5,2.25) [right] {\large $\chi_1(a)=1$};
        
        \filldraw[draw=black,fill=gray!50] (3.9,0.9)--(3.9,1.6)--(4.6,1.6)--(4.6,0.9)--cycle; 
        \node at (4.5,1.25) [right] {\large $\chi_1(a)=-1$};
        
        \filldraw[draw=black,fill=white] (4.25,0.25) circle (0.25); 
        \node at (4.5,0.25)[right] {\large $\chi_2(a)=1$};
        
        \filldraw[draw=black,fill=gray] (4.25,-0.75) circle (0.25); 
        \node at (4.5,-0.75) [right] {\large $\chi_2(a)=-1$};
        
        \draw (4.1,-1.9) -- (4.4,-1.6); 
        \draw (4.4,-1.9) -- (4.1,-1.6); 
        \node at (4.5,-1.75)[right] {\large $\chi_2(a)=0$};

    \end{tikzpicture}
    
    \caption{This circular chessboard shows the values of $\chi_1(a)$ and $\chi_2(a)$, represented by the colors of wedges and symbols, respectively.}
    \label{fig:chi}
\end{figure}
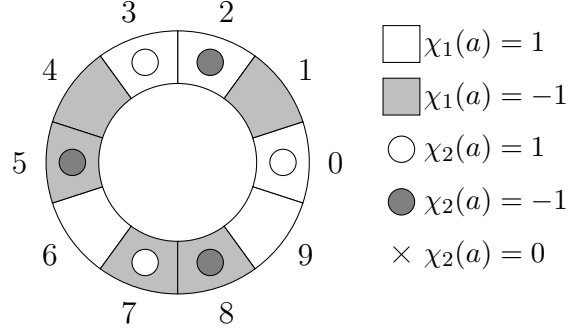

Figure~\ref{fig:chi} exhibits the same values geometrically; in particular it makes visible the
antisymmetry of $\chi_2$ under $a\mapsto a+5$ (opposite points of the circle carry opposite signs),
i.e.\ $\chi_2(5+a)=-\chi_2(a)$, which is less immediate from Table~\ref{tab:chi} alone.

The following identities are used repeatedly; they are all proved by reading Table~\ref{tab:chi}.
\begin{proposition}\label{prop:chi}
The following hold for all $a,b\in\ResT$ (with the stated conditions):
\begin{enumerate}
    \item $\chi_1(-a)=-\chi_1(a)$ for $a\not\equiv 0,5\pmod{10}$;
    \item $\chi_1(5+a)=-\chi_1(a)$;
    \item $\chi_1(5-a)=\chi_1(a)$ for $a\not\equiv 0,5\pmod{10}$;
    \item $\chi_2(-a)=\chi_2(a)$;
    \item $\chi_2(5+a)=-\chi_2(a)$;
    \item $\chi_2(5-a)=-\chi_2(a)$;
    \item if $\chi_2(a)\neq 0$, then
    $\chi_1(a-1)=\chi_2(a)$, \ $\chi_1(a-2)=-\chi_2(a)$, \ $\chi_1(a-4)=\chi_2(a)$;
    \item if $\chi_2(a)\chi_2(b)\neq 0$, then $\chi_1(b-a-4)=\chi_2(a)\chi_2(b)$.
\end{enumerate}
\end{proposition}

\begin{proof}
All statements are identities among functions on the finite set $\ResT$; they are verified directly
from Table~\ref{tab:chi}. For example, items (1) and (3) are the only ones with exceptional values: the
row $\chi_1(-a)$ equals $-\chi_1(a)$ except at $a=0$ (where $\chi_1(0)=1=\chi_1(-0)$) and at $a=5$
(where $\chi_1(5)=-1=\chi_1(-5)$); item (8) is checked over the $6\times 6$ pairs $(a,b)$ with
$\chi_2(a)\chi_2(b)\neq 0$.
\end{proof}

We write $\mex$ for the minimum excluded value of a set of nonnegative integers.

\section{Reducing strategy and single state} \label{sec:red}

We record several game equalities that shorten a position. Throughout, boards are invariant under
left--right reversal, so we state only one of each mirrored pair and use it silently in both
orientations.

\begin{lemma} \label{redrule}
Given any subpositions $w,v$, the following equalities hold:
\begin{enumerate}
    \item $[X1X w]=[X w]$; \label{red1}
    \item $[w X1X v]=[wX] \oplus [X v]$; \label{red3}
    \item $[XX1 w]=[w]$; \label{red4}
    \item $[w1XX1 v]=[w]\oplus[v]$. \label{red6}
\end{enumerate}
(The mirrors $[wX1X]=[wX]$ and $[w1XX]=[w]$ follow by reversal.)
\end{lemma}

\begin{proof}
In each equality the two sides have the same game: the cell(s) adjacent to a block $XX$ are
unplayable (placing there would complete three in a row), so the game splits or shrinks exactly as
indicated. For example, in $[XX1w]$ the single cell after $XX$ is unplayable and the block $XX$ is
inert, leaving exactly the game $[w]$; in $[w1XX1v]$ the two cells adjacent to $XX$ are unplayable,
so the position splits into the independent games $[w]$ and $[v]$. The equalities $[X1Xw]=[Xw]$ and
$[wX1Xv]=[wX]\oplus[Xv]$ hold because the middle cell of $X1X$ is unplayable, so the two sides have
identical move options.
\end{proof}

We call a position \emph{single} if it cannot be reduced by Lemma~\ref{redrule}.

\begin{remark}
A position is single if and only if it contains no subsequence $XX$ or $X1X$.
\end{remark}

The following equalities also simplify positions, but we do not count them as reductions because they
do not remove a whole block at once; they are used only in Section~\ref{sec:exceptional}.

\begin{lemma}\label{redrule3}
Given any subpositions $w,v$,
\begin{enumerate}
    \item $[X2X w]=[2X w]$; \label{redrule3_1}
    \item $[wX2X]=[wX2]$;
    \item $[wX4Xv]=[wX2] \oplus [2Xv]$.
\end{enumerate}
\end{lemma}

\begin{proof}
For (1), the two cells of the gap $2$ are adjacent to $X$'s on both sides; whichever of the two cells
is played first, the other becomes unplayable and the leftmost $X$ becomes inert, so the position is
$[2Xw]$ with the right $X$ now playing the role of the boundary. (2) is the mirror of (1). For (3),
the four cells of the gap are split as $2+2$; playing in one half cannot affect the other half because
a move adjacent to the middle leaves a gap that neither side can bridge, so the position is the sum
$[wX2]\oplus[2Xv]$.
\end{proof}

By making a move close to (at distance $1$ or $2$ from) an existing piece, one breaks the board into
shorter positions:
\[
  [w1X5v] \mapsto [w1XX4v],\ [w1X1X3v],
\]
which by Lemma~\ref{redrule} equal $[w]\oplus[3v]$ and $[w1X]\oplus[X3v]$, respectively.

Our approach to proving the main theorem is as follows. We introduce a class of ``regular positions''
(Definition~\ref{def:regular_states}) and show that a regular position is always either a
$\mathscr{P}$-position or an $\mathscr{N}$-position according to its sign. The proof is an induction
with two parts, stated precisely in Theorem~\ref{thm:reduction_mechanism} and summarized in
Section~\ref{sec:overview}. While some case analysis is inevitable, the reduction strategy keeps it
small and gives a compact hand-checked proof in place of a computer search.

\section{Main Theorem}\label{sec:main}

\subsection{Overview of the proof strategy}\label{sec:overview}

Before the technical lemmas we outline the structure of the proof. We single out a class of
\emph{regular positions} $R=R_1\cup R_{-1}$ (Definition~\ref{def:regular_states}) and prove, by strong
induction, that every position in $R_1$ has Sprague--Grundy value $0$ (it is a $\mathscr{P}$-position)
and every position in $R_{-1}$ has value $1$ (it is an $\mathscr{N}$-position). The two parts of $R$ are
indexed by the signs of the characteristic functions $\chi_1,\chi_2$ of Section~\ref{sec:prelim}; this
indexing is exactly what makes the induction algebraic.

The induction rests on a dichotomy, formalised in Theorem~\ref{thm:reduction_mechanism}. Fix a
non-terminal position $[w]\in R_s$ ($s\in\{1,-1\}$), and suppose the Next Player moves to $[w']$.
\begin{itemize}
  \item \emph{Reactive side:} if the move is not adjacent to an existing piece, the responding player
  can reply so that the position splits into a disjunctive sum $\bigoplus_j[v_j]$ of regular positions
  with sign product $\prod_j s_j=s$ (Lemmas~\ref{lem:R0_normal}, \ref{lem:R0_border},
  \ref{lem:R1_normal}, \ref{lem:R1_dist2}, \ref{lem:R1_border}, \ref{lem:Rs_normal}).
  \item \emph{Proactive side:} the Next Player always has an \emph{adjacent} move that splits the
  position into a sum of regular positions with sign product $\prod_j s_j=-s$
  (Lemmas~\ref{lem:R0_reduce}, \ref{lem:R1_reduce}, \ref{lem:Rs_reduce}).
\end{itemize}
By the induction hypothesis a regular position of sign $s_j$ has value $(1-s_j)/2$, so a sum
$\bigoplus_j[v_j]$ has value $0$ exactly when $\prod_j s_j=1$ and value $1$ exactly when
$\prod_j s_j=-1$ (equation~\eqref{eq:isomorphism}). The dichotomy then yields the two claims needed
for the main theorem, exactly as in the proof of Theorem~\ref{thm:main_SG}.

The reason we study the products $\chi_1(a)\chi_1(b)$ and $\chi_2(a)\chi_2(b)$ is therefore this: after
a move and its reduction, the sign product of the resulting components is computed as a product of
values of $\chi_1,\chi_2$, and matching that product with $s$ is exactly what places the reduced
position in the correct $\mathscr{P}/\mathscr{N}$ class.

\subsection{An algebraic lemma}

We first prove the algebraic identity underlying all interior responses.

\begin{lemma}\label{lem:alpha}
For any $a,b\in\mathbb{Z}_{10}$, at least one of the following two equalities holds:
\begin{align}
    \chi_1(a)\chi_1(b) &= \chi_1(a+b+4), \label{eq:alpha1} \\
    \chi_1(a+1)\chi_1(b-1) &= \chi_1(a+b+4). \label{eq:alpha2}
\end{align}
\end{lemma}

\begin{proof}
A visual proof is given in Figure~\ref{fig:lemma}; we now give an algebraic one. Recall from
Section~\ref{sec:prelim} that $x\mapsto x/3$ coincides with $x\mapsto 7x$ on
$\mathbb{Z}_{10}$ (since $3\cdot 7\equiv 1$); its inverse is $x\mapsto 3x$. Let
$I=\{-2,-1,0,1,2\}$ and define $\sigma:\mathbb{Z}_{10}\to\{1,-1\}$ by $\sigma(y)=1$ if
$y\pmod{10}\in I$ and $\sigma(y)=-1$ if $y\pmod{10}\in I+5=\{3,4,5,6,7\}$. Then
$\chi_1(x)=\sigma(7x-2)$. Write $7x-2=c_x+d_x$ uniquely with $c_x\in\{0,5\}$ and $d_x\in I$; then
$\chi_1(x)=\sigma(c_x)$, because $\sigma(c_x+d_x)=\sigma(c_x)\sigma(d_x)$ and $\sigma(d_x)=1$ for
$d_x\in I$.

For $y_1,y_2\in\mathbb{Z}_{10}$ decomposed as $y_i=c_i+d_i$ with $c_i\in\{0,5\}$, $d_i\in I$,
\begin{equation}\label{eq:sigma}
  \sigma(y_1+y_2)=\sigma(c_1)\sigma(c_2)\sigma(d_1+d_2).
\end{equation}
Indeed $c_1+c_2\in\{0,5,10\}$ and $\sigma(c_1)\sigma(c_2)=\sigma(c_1+c_2)$, so it remains to verify
$\sigma(c_1+c_2+d_1+d_2)=\sigma(c_1+c_2)\sigma(d_1+d_2)$; this is clear when $c_1+c_2\equiv 0$ and, when
$c_1+c_2\equiv 5$, reduces to $\sigma(5+t)=-\sigma(t)$, which holds because adding $5$ interchanges $I$
and $I+5$.

Now $\chi_1(a+b+4)=\sigma((7a-2)+(7b-2))$, so \eqref{eq:sigma} with $y_1=7a-2$, $y_2=7b-2$ gives
\begin{equation}\label{eq:chi_sum}
  \chi_1(a+b+4)=\chi_1(a)\chi_1(b)\sigma(d_a+d_b),
\end{equation}
where $7a-2=c_a+d_a$ and $7b-2=c_b+d_b$. If $d_a+d_b\in I$ then $\sigma(d_a+d_b)=1$ and
\eqref{eq:alpha1} holds. Otherwise $d_a,d_b\in\{1,2\}$ or $d_a,d_b\in\{-1,-2\}$, and in either case
$d_a+d_b\in I+5$, so
\begin{equation}\label{eq:chi_neg}
  \chi_1(a+b+4)=-\chi_1(a)\chi_1(b).
\end{equation}
We then shift the arguments. Since $\chi_1(a+1)=\sigma((7a-2)-3)$ and
$\chi_1(b-1)=\sigma((7b-2)+3)$,
\[
  \chi_1(a+1)\chi_1(b-1)=\chi_1(a)\chi_1(b)\,\sigma(d_a-3)\sigma(d_b+3).
\]
A direct check over the two subcases gives $\sigma(d_a-3)\sigma(d_b+3)=-1$: if
$d_a,d_b\in\{1,2\}$ then $d_a-3\in\{-2,-1\}\subset I$ and $d_b+3\in\{4,5\}\subset I+5$; if
$d_a,d_b\in\{-1,-2\}$ then $d_a-3\in\{-4,-5\}\subset I+5$ and $d_b+3\in\{1,2\}\subset I$. Thus
$\chi_1(a+1)\chi_1(b-1)=-\chi_1(a)\chi_1(b)=\chi_1(a+b+4)$ by \eqref{eq:chi_neg}, which is
\eqref{eq:alpha2}.
\end{proof}

\begin{figure}[htbp]

    \centering
    \begin{tikzpicture}[scale=0.7]

        \foreach \x/\y in {
            0/1, 2/1, 3/1, 6/1, 9/1, 
            0/4, 2/4, 3/4, 6/4, 9/4, 
            0/5, 2/5, 3/5, 6/5, 9/5, 
            0/7, 2/7, 3/7, 6/7, 9/7, 
            0/8, 2/8, 3/8, 6/8, 9/8, 
            1/0, 4/0, 5/0, 7/0, 8/0, 
            1/2, 4/2, 5/2, 7/2, 8/2, 
            1/3, 4/3, 5/3, 7/3, 8/3, 
            1/6, 4/6, 5/6, 7/6, 8/6, 
            1/9, 4/9, 5/9, 7/9, 8/9
        } {
            \fill[gray!50] (\x,\y)--(\x,\y+1)--(\x+1,\y+1)--(\x+1,\y)--cycle;
        }
        
        \foreach \x in {1, 4, 5, 7, 8}{
            \fill[gray!50] (\x,-0.5)--(\x,0)--(\x+1,0)--(\x+1,-0.5)--cycle;
            \fill[gray!50] (\x,10.5)--(\x,10)--(\x+1,10)--(\x+1,10.5)--cycle;
            \fill[gray!50] (-0.5,\x)--(0,\x)--(0,\x+1)--(-0.5,\x+1)--cycle;
            \fill[gray!50] (10.5,\x)--(10,\x)--(10,\x+1)--(10.5,\x+1)--cycle;
        }

        \foreach \x in {0, 1, 2, 3, 4, 5, 6, 7, 8, 9}{
            \node at (\x+0.5,-0.75) {\large \x};
            \node at (-0.75,\x+0.5) {\large \x};
        }

        \draw (-0.5, -0.5) grid (10.5, 10.5);
        \draw[thick] (0,0)--(0,10)--(10,10)--(10,0)--cycle;

        \foreach \x/\y in {
            6/0, 8/0, 9/0, 2/0, 5/0, 
            7/9, 9/9, 0/9, 3/9, 6/9, 
            8/8, 0/8, 1/8, 4/8, 7/8,
            9/7, 1/7, 2/7, 5/7, 8/7,
            0/6, 2/6, 3/6, 6/6, 9/6,
            1/5, 3/5, 4/5, 7/5, 0/5,
            2/4, 4/4, 5/4, 8/4, 1/4,
            3/3, 5/3, 6/3, 9/3, 2/3,
            4/2, 6/2, 7/2, 0/2, 3/2,
            5/1, 7/1, 8/1, 1/1, 4/1
        } {
            \filldraw[draw=black,fill=white] (\x+0.5,\y+0.5) circle (0.25);
        }
        
        \foreach \x/\y in {
            6/5, 8/5, 9/5, 2/5, 5/5, 
            7/4, 9/4, 0/4, 3/4, 6/4, 
            8/3, 0/3, 1/3, 4/3, 7/3, 
            9/2, 1/2, 2/2, 5/2, 8/2,
            0/1, 2/1, 3/1, 6/1, 9/1,
            1/0, 3/0, 4/0, 7/0, 0/0,
            2/9, 4/9, 5/9, 8/9, 1/9,
            3/8, 5/8, 6/8, 9/8, 2/8,
            4/7, 6/7, 7/7, 0/7, 3/7,
            5/6, 7/6, 8/6, 1/6, 4/6
        } {
            \filldraw[draw=black,fill=gray] (\x+0.5,\y+0.5) circle (0.25);
        }
        
        \filldraw[draw=black,fill=white] (11.4,8.9)--(11.4,9.6)--(12.1,9.6)--(12.1,8.9)--cycle; 
        \node at (12,9.25)[right] {\large $\chi_1(a)\chi_1(b)=1$};
        
        \filldraw[draw=black,fill=gray!50] (11.4,7.9)--(11.4,8.6)--(12.1,8.6)--(12.1,7.9)--cycle; 
        \node at (12,8.25) [right] {\large $\chi_1(a)\chi_1(b)=-1$};
        
        \filldraw[draw=black,fill=white] (11.75,7.25) circle (0.25); 
        \node at (12,7.25) [right] {\large $\chi_1(a+b+4)=1$};
        
        \filldraw[draw=black,fill=gray] (11.75,6.25) circle (0.25); 
        \node at (12,6.25) [right] {\large $\chi_1(a+b+4)=-1$};
        
    \end{tikzpicture}

    \caption{This toroidal chessboard (with opposite edges glued together) shows the values of $\chi_1(a)\chi_1(b)$ and $\chi_1(a+b+4)$, illustrating Lemma~\ref{lem:alpha}.
        Each piece is either the same color as the square it is on (so $\chi_1(a)\chi_1(b)=\chi_1(a+b+4)$), or the same color as the square at its lower right corner (so $\chi_1(a+1)\chi_1(b-1)=\chi_1(a+b+4)$).
    }
    \label{fig:lemma}
\end{figure}
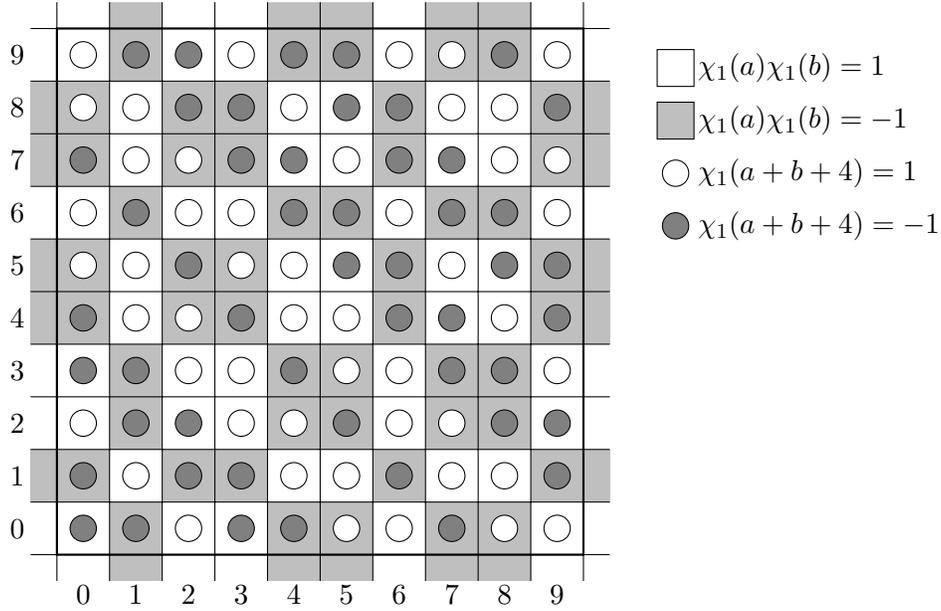

\begin{definition}\label{def:regular_states}
The set of regular positions $R=R_1\cup R_{-1}$ is classified as follows; the subscripts record the
eventual Sprague--Grundy value and the superscripts record the number of pieces.

\begin{enumerate}
  \item \emph{Class $R_1$} (positions of value $0$):
  \begin{itemize}
    \item $R_1^{(0)}$: empty boards $[a]$ with $\chi_1(a)=1$;
    \item $R_1^{(1)}$: single-piece boards $[a X b]$ with $\chi_2(a)\chi_2(b)=1$;
    \item $R_1^{*}$: the exceptional boundary states $[X\bar 6]$, $[\bar 6 X]$, $[X\bar 4 X]$.
  \end{itemize}
  \item \emph{Class $R_{-1}$} (positions of value $1$):
  \begin{itemize}
    \item $R_{-1}^{(0)}$: empty boards $[a]$ with $\chi_1(a)=-1$;
    \item $R_{-1}^{(1)}$: single-piece boards $[a X b]$ with $\chi_2(a)\chi_2(b)=-1$;
    \item $R_{-1}^{*}$: the exceptional boundary states $[X\bar 1]$, $[\bar 1 X]$, $[X\bar 9 X]$.
  \end{itemize}
\end{enumerate}
\end{definition}

All six subfamilies are invariant under left--right reflection, and we use this symmetry to consider
only one of two mirrored cases. Whenever a reduction of Lemma~\ref{redrule} removes a segment, the
removed segment is treated as a trivial component of value $0$.

\subsection{Empty board states $R_s^{(0)}$}\label{sec:empty}

Let $[a]\in R_s^{(0)}$ be an empty board of length $a\ge 1$. A move places one piece, giving
$[a_1 X a_2]$ with $a_1+a_2=a-1$. There are three cases: (1) \emph{terminal}, $a=1$; (2) \emph{boundary},
exactly one of $a_1,a_2$ is $0$; (3) \emph{interior}, $a_1\ge 1$ and $a_2\ge 1$.

\begin{lemma}\label{lem:R0_border}
Suppose $[a]\in R_s^{(0)}$ transitions to $[X(a-1)]$ by a boundary move. Then the responding player can
place a piece at distance $1$ or $2$ from the new piece so that the resulting position lies in $R_s$.
\end{lemma}

\begin{proof}
The two candidate replies are $[XX(a-2)]\cong[a-3]$ and $[X1X(a-3)]\cong[X(a-3)]$. If
$\chi_1(a-3)=\chi_1(a)=s$ the first candidate lands in $R_s^{(0)}$; this holds for all residues except
$a\equiv 0,5\pmod{10}$. In those two exceptional classes the second candidate $[0X(a-3)]$ satisfies
$\chi_2(0)\chi_2(a-3)=s$: for $a\equiv 0$ we obtain $[0X7]$ with $\chi_2(0)\chi_2(7)=1=s$, and for
$a\equiv 5$ we obtain $[0X2]$ with $\chi_2(0)\chi_2(2)=-1=s$. The small cases $a=2,3$ are covered by
the first candidate: for $a=2$, $[XX0]=[0]$ is terminal of value $0$ and
$\chi_1(-1)=\chi_1(9)=1=\chi_1(2)$; for $a=3$, $[a-3]=[0]$ and $\chi_1(0)=1=\chi_1(3)$.
\end{proof}

\begin{lemma}\label{lem:R0_normal}
Suppose $[a]\in R_s^{(0)}$ transitions to $[a_1 X a_2]$ by an interior move ($a_1,a_2\ge 1$). Then the
responding player can place a piece adjacent to the new one so that the position reduces to a sum
$[v_1]\oplus[v_2]$ of regular positions with sign product $s$.
\end{lemma}

\begin{proof}
Placing adjacent to the new piece, the responding player has two candidate moves and the following
reductions (Lemma~\ref{redrule}):
\begin{align}
  \text{if } a_1\ge 2:\quad [(a_1-1)XX a_2] &\cong [(a_1-2)]\oplus[(a_2-1)], \label{eq:r0_left}\\
  \text{if } a_2\ge 2:\quad [a_1 XX(a_2-1)] &\cong [(a_1-1)]\oplus[(a_2-2)], \label{eq:r0_right}\\
  \text{if } a_1=1:\quad [XX a_2] &\cong [(a_2-1)], \label{eq:r0_a1}\\
  \text{if } a_2=1:\quad [a_1 XX] &\cong [(a_1-1)]. \label{eq:r0_a2}
\end{align}
We need one of these to have sign $s=\chi_1(a)=\chi_1(a_1+a_2+1)$.

If $a_1\ge 2$ and $a_2\ge 2$, then by Lemma~\ref{lem:alpha} (with $x=a_1-2$, $y=a_2-1$) at least one of
the signs of \eqref{eq:r0_left}, \eqref{eq:r0_right} equals $\chi_1(a_1+a_2+1)=s$.

If $a_1=1$ and $a_2\ge 2$, we use \eqref{eq:r0_a1} or \eqref{eq:r0_right}, whose signs are
$\chi_1(a_2-1)$ and $\chi_1(a_2-2)$. The needed equality is $\chi_1(a_2-1)=s=\chi_1(a_2+2)$ or
$\chi_1(a_2-2)=\chi_1(a_2+2)$; a direct check from the table in Proposition~\ref{prop:chi} shows that
for every residue $a_2$ at least one of these two equalities holds (the first fails only for
$a_2\equiv 3,8$, where the second holds). The case $a_2=1$, $a_1\ge 2$ is symmetric.

If $a_1=a_2=1$, then $a=3$ and $s=\chi_1(3)=1$; the adjacent move $[1 X 1]\to[XX1]\cong[0]$ has
sign $\chi_1(0)=1=s$.
\end{proof}

\begin{lemma}\label{lem:R0_reduce}
For any empty board $[a]\in R_s^{(0)}$ there is a move to a position in $R_{-s}^{(1)}\cup R_{-s}^{*}$.
\end{lemma}

\begin{proof}
Enumerate the residues (all congruences modulo $10$). For $s=1$ ($a\in\{\bar0,\bar2,\bar3,\bar6,\bar9\}$):
\begin{align*}
  [\bar0]&\mapsto[\bar2 X\bar7]\in R_{-1}^{(1)}, &
  [\bar3]&\mapsto[X\bar2]\in R_{-1}^{(1)}, &
  [\bar6]&\mapsto[X\bar5]\in R_{-1}^{(1)}, \\
  [\bar9]&\mapsto[X\bar8]\in R_{-1}^{(1)}, &
  [\bar2]&\mapsto[X\bar1]\in R_{-1}^{*}.
\end{align*}
For $s=-1$ ($a\in\{\bar1,\bar4,\bar5,\bar7,\bar8\}$):
\begin{align*}
  [\bar5]&\mapsto[\bar2 X\bar2]\in R_{1}^{(1)}, &
  [\bar8]&\mapsto[\bar2 X\bar5]\in R_{1}^{(1)}, &
  [\bar1]&\mapsto[X\bar0]\in R_{1}^{(1)}, \\
  [\bar4]&\mapsto[X\bar3]\in R_{1}^{(1)}, &
  [\bar7]&\mapsto[X\bar6]\in R_{1}^{*}.
\end{align*}
In each line the target lies in the stated class by direct evaluation of $\chi_1,\chi_2$.
\end{proof}

\subsection{General single-piece states $R_s^{(1)}$}\label{sec:single}

Consider $[a X_1 b]\in R_s^{(1)}$, so $\chi_2(a)\chi_2(b)=s\neq 0$; in particular
$\chi_2(a),\chi_2(b)\in\{1,-1\}$, hence $a,b\not\equiv 1,4,6,9\pmod{10}$. By symmetry assume the Next
Player places $X_2$ in the right segment $b$, at distance $b_1$ from $X_1$: the position is
$[a X_1 b_1 X_2 b_2]$ with $b_1+b_2=b-1$. We distinguish four cases: $b_1=0$ (adjacent), $b_1=1$
(distance two), $b_1\ge 2$ with $b_2\ge 1$ (interior), and $b_1\ge 2$ with $b_2=0$ (boundary).

\begin{lemma}\label{lem:R1_reduce}
If the Next Player places $X_2$ adjacent to $X_1$ ($b_1=0$), the position reduces to a disjunctive sum
$[v_1]\oplus[v_2]$ of regular positions with sign product $-s$.
\end{lemma}

\begin{proof}
The position is $[a X_1 X_2 (b-1)]$. If $a\ge 1$ and $b\ge 2$ it reduces by Lemma~\ref{redrule}(4) to
$[(a-1)]\oplus[(b-2)]$, of sign
\[
  \chi_1(a-1)\chi_1(b-2)=\chi_2(a)\cdot(-\chi_2(b))=-s
\]
by Proposition~\ref{prop:chi}(7). If $a=0$ the same move gives $[XX(b-1)]\cong[b-2]$ of sign
$\chi_1(b-2)=-\chi_2(b)=-s$; if $b=1$ it gives $[(a-1)]$ of sign $\chi_1(a-1)=\chi_2(a)=-s$ (note
$b=1$ is excluded since $\chi_2(1)=0$).
\end{proof}

\begin{lemma}\label{lem:R1_dist2}
If the Next Player places $X_2$ at distance two from $X_1$ ($b_1=1$), then either the position reduces
to a regular position of sign $-s$ (when $b=2$), or the responding player can reply so that the
position reduces to a regular position of sign $s$ (when $b\ge 3$).
\end{lemma}

\begin{proof}
The position is $[a X_1 1 X_2 (b-2)]$. If $b=2$ it is $[a X_1 1 X_2]$, which reduces by
Lemma~\ref{redrule}(2) to $[aX]\oplus[X]\cong[aX]$; its sign is $\chi_2(a)$, while
$s=\chi_2(a)\chi_2(2)=-\chi_2(a)$, so the sign is $-s$.

If $b\ge 3$ the responding player places $X_3$ adjacent to $X_2$ on the right, obtaining
$[a X_1 1 X_2 X_3 (b-3)]$. Lemma~\ref{redrule}(2) splits this as $[aX]\oplus[XX(b-3)]$, and
Lemma~\ref{redrule}(3) reduces the second summand to $[(b-4)]$ (for $b=3$ it is the terminal board
$[0]$). The sign is $\chi_2(a)\chi_1(b-4)=\chi_2(a)\chi_2(b)=s$ by
Proposition~\ref{prop:chi}(7), since $\chi_1(b-4)=\chi_2(b)$.
\end{proof}

\begin{lemma}\label{lem:R1_normal}
If the Next Player makes an interior move ($b_1\ge 2$, $b_2\ge 1$), the responding player can place a
piece adjacent to $X_2$ to reduce to $[v_1]\oplus[v_2]$ with sign product $s$.
\end{lemma}

\begin{proof}
The two candidate replies adjacent to $X_2$ give, respectively,
\[
  [a X(b_1-2)]\oplus[(b_2-1)],\qquad [a X(b_1-1)]\oplus[(b_2-2)],
\]
with sign products $\chi_2(a)\chi_2(b_1-2)\chi_1(b_2-1)$ and
$\chi_2(a)\chi_2(b_1-1)\chi_1(b_2-2)$. Since $\chi_2(b_1-1)$ and $\chi_2(b_1-2)$ are not both zero,
choose $i\in\{1,2\}$ with $\chi_2(b_1-i)\neq 0$. Writing $b=b_1+b_2+1$ we have
$b_2-(3-i)=b-(b_1-i)-4$, and by Proposition~\ref{prop:chi}(8),
\[
  \chi_1\bigl(b_2-(3-i)\bigr)=\chi_1\bigl(b-(b_1-i)-4\bigr)=\chi_2(b_1-i)\,\chi_2(b).
\]
Multiplying by $\chi_2(a)$ and using $\chi_2(b_1-i)^2=1$ gives $\chi_2(a)\chi_2(b)=s$, as required.
\end{proof}

\begin{lemma}\label{lem:R1_border}
If the Next Player makes a boundary move ($b_1\ge 2$, $b_2=0$, so $b\ge 3$), the responding player can
reduce to a regular position (or sum) of sign $s$.
\end{lemma}

\begin{proof}
The position is $[a X_1(b-1)X_2]$; since $\chi_2(b)\neq 0$ we have $b\equiv 0,2,3,5,7,8\pmod{10}$.
\begin{itemize}
  \item If $b\equiv 0,3,5,8$, place $X_3$ adjacent to $X_2$ on the left. This yields
  $[a X(b-2)X_3 X_2]\cong[a X(b-3)]$ (Lemma~\ref{redrule}(3)), a single-piece board. Its sign is
  $\chi_2(a)\chi_2(b-3)=s$, because $\chi_2(b-3)=\chi_2(b)$ for these residues.
  \item If $b\equiv 7$, place $X_3$ at distance two from $X_1$: the position
  $[a X_1 1 X_3(b-3)X_2]$ reduces by Lemma~\ref{redrule}(2) to $[aX]\oplus[X(b-3)X]$; here
  $b-3\equiv 4$, so the second component is $[X\bar 4 X]\in R_1^{*}$, and the sign product is
  $\chi_2(a)\cdot 1=\chi_2(a)\chi_2(7)=s$.
  \item If $b\equiv 2$ (so $b\ge 12$), the same distance-two reply gives $[aX]\oplus[X(b-3)X]$ with
  $b-3\equiv 9$, i.e.\ $[X\bar 9 X]\in R_{-1}^{*}$; the sign product is
  $\chi_2(a)\cdot(-1)=\chi_2(a)\chi_2(2)=s$.
\end{itemize}
\end{proof}

\subsection{Exceptional boundary states $R_s^{*}$}\label{sec:exceptional}

Up to reflection the exceptional states are the four residue classes $[X\bar6]$, $[X\bar4X]\in R_1^{*}$
and $[X\bar1]$, $[X\bar9X]\in R_{-1}^{*}$. A move is \emph{adjacent} if the new piece is placed next to
an existing piece, and \emph{interior} otherwise (distance at least two).

\begin{lemma}\label{lem:Rs_reduce}
If the Next Player makes an adjacent move from an exceptional state, the position reduces to a regular
position (or sum) of sign $-s$.
\end{lemma}

\begin{proof}
Check the four classes:
\[
  [X\bar6]\to [XX\bar5]\cong[\bar4]\in R_{-1}^{(0)},\qquad
  [X\bar4X]\to [XX\bar3X]\cong[\bar2X]\in R_{-1}^{(1)},
\]
\[
  [X\bar1]\to [XX\bar0]\cong[\bar9]\in R_{1}^{(0)},\qquad
  [X\bar9X]\to [XX\bar8X]\cong[\bar7X]\in R_{1}^{(1)}.
\]
The first two have initial sign $s=1$ and reduce to sign $-1$; the last two have $s=-1$ and reduce to
sign $1$.
\end{proof}

\begin{lemma}\label{lem:Rs_normal}
If the Next Player makes an interior move (distance at least two) from an exceptional state, the
responding player can reply so that the position reduces to a sum of regular positions of sign $s$.
\end{lemma}

\begin{proof}
We treat $[X\bar6]$ and $[X\bar4X]$; $[X\bar1]$ and $[X\bar9X]$ follow by reversing all signs.

\emph{The state $[X\bar6]=[X(10k+6)]$.} An interior move gives $[X_1 a X_2 b]$ with $a+b=10k+5$ and
$a\ge 1$. The responding player places a piece adjacent to $X_1$ or $X_2$ as listed in
Table~\ref{tab:exceptional6}; in each row the position reduces to a disjunctive sum of regular
positions of sign product $1=s$.
\begin{table}[htpb]
  \centering
  \caption{Response to an interior move from $[X\bar6]$; every row has sign product $1$.}
  \label{tab:exceptional6}
  \small
  \begin{tabular}{c|l|l}
    $a\bmod 10$ & response & reduced components and their classes \\ \hline
    $0$ & adjacent to $X_1$ & $[\bar8 X\bar5]\in R_1^{(1)}$ \\
    $1$ & adjacent right of $X_2$ & $[\bar2]\in R_1^{(0)}$ \\
    $2$ & adjacent left of $X_2$ & $[\bar2]\in R_1^{(0)}$ \\
    $3$ & adjacent left of $X_2$ & $[X\bar1]\in R_{-1}^{*}$, $[\bar1]\in R_{-1}^{(0)}$ \\
    $4$ & adjacent right of $X_2$ & $[X\bar3]\in R_1^{(1)}$, $[\bar9]\in R_1^{(0)}$ \\
    $5$ & adjacent to $X_1$ & $[\bar3X]\in R_1^{(1)}$ \\
    $6$ & adjacent right of $X_2$ & $[X\bar5]\in R_{-1}^{(1)}$, $[\bar7]\in R_{-1}^{(0)}$ \\
    $7$ & adjacent to $X_1$ & $[\bar5 X\bar8]\in R_1^{(1)}$ \\
    $8$ & adjacent left of $X_2$ & $[X\bar6]\in R_1^{*}$, $[\bar6]\in R_1^{(0)}$ \\
    $9$ & adjacent right of $X_2$ & $[X\bar8]\in R_{-1}^{(1)}$, $[\bar4]\in R_{-1}^{(0)}$ \\
  \end{tabular}
\end{table}
Each entry is obtained by applying the reductions of Lemmas~\ref{redrule} and~\ref{redrule3}; the
signs multiply to $1$ in every row (for instance the $a\equiv 3$ row gives $(-1)\cdot(-1)=1$).

\emph{The state $[X\bar4X]=[X(10k+4)X]$.} An interior move gives $[X a X_2 b X]$ with $a+b=10k+3$,
$a\ge 1$, $b\ge 1$. The responding player places a piece adjacent to $X_2$ as listed in
Table~\ref{tab:exceptional4X}; again every row reduces to a sum of regular positions of sign product
$1=s$.
\begin{table}[htpb]
  \centering
  \caption{Response to an interior move from $[X\bar4X]$; every row has sign product $1$.}
  \label{tab:exceptional4X}
  \small
  \begin{tabular}{c|l|l}
    $a\bmod 10$ & response & reduced components and their classes \\ \hline
    $0$ & adjacent left of $X_2$ & $[X\bar8]\in R_{-1}^{(1)}$, $[\bar2X]\in R_{-1}^{(1)}$ \\
    $1$ & adjacent right of $X_2$ & $[X]\in R_1^{(1)}$ \\
    $2$ & adjacent left of $X_2$ & $[X]\in R_1^{(1)}$ \\
    $3$ & adjacent right of $X_2$ & $[X\bar2]\in R_{-1}^{(1)}$, $[\bar8X]\in R_{-1}^{(1)}$ \\
    $4$ & adjacent right of $X_2$ & $[X\bar3]\in R_1^{(1)}$, $[\bar7X]\in R_1^{(1)}$ \\
    $5$ & adjacent left of $X_2$ & $[X\bar3]\in R_1^{(1)}$, $[\bar7X]\in R_1^{(1)}$ \\
    $6$ & adjacent right of $X_2$ & $[X\bar5]\in R_{-1}^{(1)}$, $[\bar5X]\in R_{-1}^{(1)}$ \\
    $7$ & adjacent left of $X_2$ & $[X\bar5]\in R_{-1}^{(1)}$, $[\bar5X]\in R_{-1}^{(1)}$ \\
    $8$ & adjacent right of $X_2$ & $[X\bar7]\in R_1^{(1)}$, $[\bar3X]\in R_1^{(1)}$ \\
    $9$ & adjacent left of $X_2$ & $[X\bar7]\in R_1^{(1)}$, $[\bar3X]\in R_1^{(1)}$ \\
  \end{tabular}
\end{table}
In every row the listed classes are regular positions and their signs multiply to $1$; note that
$[X\bar4X]\in R_1^{*}$ and $[X\bar9X]\in R_{-1}^{*}$ have signs $1$ and $-1$ respectively where they
arise as components (they do so in the $a\equiv 8,9$ rows of Table~\ref{tab:exceptional6} via the
exceptional states $[X\bar6]$, and in the reductions of the exceptional states themselves).
\end{proof}

\subsection{Synthesis and the Sprague--Grundy evaluation}\label{sec:synthesis}

\begin{theorem}\label{thm:reduction_mechanism}
Let $[w]\in R_s$ be a non-terminal regular position.
\begin{enumerate}
  \item \emph{Reactive reduction:} for every move $[w]\mapsto[w']$, the responding player has a reply
  $[w']\mapsto\bigoplus_j[v_j]$ with each $[v_j]\in R_{s_j}$ and $\prod_j s_j=s$.
  \item \emph{Proactive reduction:} the Next Player has a move $[w]\mapsto\bigoplus_j[u_j]$ with each
  $[u_j]\in R_{s_j}$ and $\prod_j s_j=-s$.
\end{enumerate}
\end{theorem}

\begin{proof}
For (1), the interior and boundary moves are handled by Lemmas~\ref{lem:R0_normal},
\ref{lem:R0_border}, \ref{lem:R1_normal}, \ref{lem:R1_dist2}, \ref{lem:R1_border},
\ref{lem:Rs_normal}. For an adjacent move, Lemmas~\ref{lem:R1_reduce} and~\ref{lem:Rs_reduce} show that
$[w']$ itself already splits into regular components of sign product $-s$; the responding player then
applies the proactive mechanism (2) to a single component, which flips its sign and restores the total
product to $s$. For (2), empty boards are covered by Lemma~\ref{lem:R0_reduce}, and single-piece and
exceptional boards by their adjacent moves (Lemmas~\ref{lem:R1_reduce} and~\ref{lem:Rs_reduce}).
\end{proof}

\begin{theorem}\label{thm:main_SG}
The regular positions satisfy $\mathcal{G}(R_1)=0$ and $\mathcal{G}(R_{-1})=1$. Consequently an empty
board of length $n$ has value $0$ if $\chi_1(n)=1$ and $1$ if $\chi_1(n)=-1$.
\end{theorem}

\begin{proof}
We prove by strong induction on the number of cells that every regular position of sign $s$ has value
$(1-s)/2$. The terminal base cases are the empty board $[0]$ ($\chi_1(0)=1$, value $0$) and the
one-piece board $[X]$ (value $0$); they conform to the claim. Assume the claim for all shorter regular
positions.

By the induction hypothesis a regular position of sign $s_j$ has value $(1-s_j)/2\in\{0,1\}$, so for a
sum of regular positions the nim-sum is governed by the sign product:
\begin{equation}\label{eq:isomorphism}
  \bigoplus_j \mathcal{G}([v_j])=0\iff \prod_j s_j=1,
  \qquad
  \bigoplus_j \mathcal{G}([v_j])=1\iff \prod_j s_j=-1 .
\end{equation}

\emph{Case 1: $[w]\in R_1$.} For any move $[w]\mapsto[w']$, Theorem~\ref{thm:reduction_mechanism}(1)
gives a reply to $\bigoplus_j[v_j]$ with $\prod_j s_j=1$, so by \eqref{eq:isomorphism} that sum has
value $0$. Hence every option $[w']$ has an option of value $0$, so
$\mathcal{G}([w'])=\mex\{0,\dots\}\neq 0$ for every $[w']$, and
$\mathcal{G}([w])=\mex_{w'}\{\mathcal{G}([w'])\}=0$.

\emph{Case 2: $[w]\in R_{-1}$.} For any move $[w]\mapsto[w']$,
Theorem~\ref{thm:reduction_mechanism}(1) gives a reply to $\bigoplus_j[v_j]$ with $\prod_j s_j=-1$, of
value $1$; hence $\mathcal{G}([w'])=\mex\{1,\dots\}\neq 1$ for every $[w']$. On the other hand,
Theorem~\ref{thm:reduction_mechanism}(2) gives a move to a sum with $\prod_j s_j=-(-1)=1$, of value
$0$. Therefore $\mathcal{G}([w])=\mex\{0,\dots\}=1$.

This completes the induction. The empty board $[n]$ lies in $R_{\chi_1(n)}^{(0)}$, giving the main
theorem.
\end{proof}

\section{Open Problems and Further Extensions}

Although the initial empty-board problem for Inverse Treblecross is resolved, several structural
patterns extend beyond our current framework. Since the ``reduction move'' strategy relies on
immediate parity responses, it is insufficient to characterize these broader phenomena. In this
section we propose conjectures and extensions, which likely require more sophisticated tools.

\subsection{Coverage of Regular Positions}

The regular positions defined in the main theorem contain at most one piece, with the two exceptions
$[X \bar 4 X]$ and $[X \bar 9 X]$. Computational evidence suggests that $R$ nearly exhausts the simple
states of SG value $0$ or $1$.

\begin{conjecture}\label{conj:cover}
For all single-piece boards $[a X b]$ with $a,b\ge 2$, the regular set $R$ contains every state with
$\mathcal{G}\le 1$, except the isolated states
\[
    [1 X 1], \quad [1 X 4], \quad \text{and} \quad [4 X 4]
\]
(up to left--right reflection), which have value $1$ and lie outside $R$.
\end{conjecture}

\begin{remark}
The restriction $a,b\ge 2$ is necessary: for example $[1X0]=[1X]$ has value $1$ but
$\chi_2(1)\chi_2(0)=0$, so it lies outside $R$; the same applies to $[0X1]$. Such boundary states of
one piece are precisely the exceptional classes $R^{*}$ already handled separately.
\end{remark}

These exceptional states are isolated anomalies: up to left--right reflection they are the unique
known states in the residue classes ($[\bar 1 X \bar 1]$, $[\bar 1 X \bar 4]$, $[\bar 4 X \bar 4]$)
of value $1$, and empirical data indicate that all other sufficiently long states in these classes
have $\mathcal{G}>1$.

\subsection{Asymptotic Periodicity of Single-piece States}

For the general single-piece configuration $[a X b]$, empirical calculations (partially displayed in
Table \ref{tab:single_piece}) reveal an asymptotic periodicity mirroring that of the empty board.

\begin{table}[htpb]
    \centering
    \caption{SG values $\mathcal{G}([a X b])$ for small lengths $a$ and $b$.}
    \label{tab:single_piece}
    \footnotesize
    \begin{tabular}{c|cccccccccccccccc}
         $a \setminus b$ & 0 & 1 & 2 & 3 & 4 & 5 & 6 & 7 & 8 & 9 & 10 & 11 & 12 & 13 & 14 & 15 \\
        \hline
          0 & 0 & 1 & 1 & 0 & 3 & 1 & 0 & 0 & 1 & 2 & 0 & 1 & 1 & 0 & 3 & 1 \\
          1 & 1 & 1 & 2 & 3 & 1 & 2 & 3 & 3 & 2 & 3 & 3 & 2 & 2 & 3 & 2 & 2 \\
          2 & 1 & 2 & 0 & 1 & 2 & 0 & 3 & 1 & 0 & 3 & 1 & 2 & 0 & 1 & 2 & 0 \\
          3 & 0 & 3 & 1 & 0 & 3 & 1 & 2 & 0 & 1 & 2 & 0 & 3 & 1 & 0 & 3 & 1 \\
          4 & 3 & 1 & 2 & 3 & 1 & 2 & 3 & 3 & 2 & 4 & 3 & 2 & 2 & 3 & 5 & 2 \\
          5 & 1 & 2 & 0 & 1 & 2 & 0 & 3 & 1 & 0 & 3 & 1 & 2 & 0 & 1 & 2 & 0 \\
          6 & 0 & 3 & 3 & 2 & 3 & 3 & 2 & 2 & 3 & 2 & 2 & 3 & 3 & 2 & 3 & 3 \\
          7 & 0 & 3 & 1 & 0 & 3 & 1 & 2 & 0 & 1 & 2 & 0 & 3 & 1 & 0 & 3 & 1 \\
          8 & 1 & 2 & 0 & 1 & 2 & 0 & 3 & 1 & 0 & 3 & 1 & 2 & 0 & 1 & 2 & 0 \\
          9 & 2 & 3 & 3 & 2 & 4 & 3 & 2 & 2 & 3 & 5 & 2 & 3 & 3 & 2 & 4 & 3 \\
         10 & 0 & 3 & 1 & 0 & 3 & 1 & 2 & 0 & 1 & 2 & 0 & 3 & 1 & 0 & 3 & 1 \\
         11 & 1 & 2 & 2 & 3 & 2 & 2 & 3 & 3 & 2 & 3 & 3 & 2 & 2 & 3 & 2 & 2 \\
    \end{tabular}
    \normalsize
\end{table}

\begin{conjecture}\label{conj:preperiod}
For any fixed $a$, the SG sequence of $[a X b]$ with respect to $b$ is eventually periodic with period
$10$. Furthermore, there is a minimal preperiod $b_a$ such that for all $b>b_a$,
\[
    \mathcal{G}([a X (b+5)]) = \mathcal{G}([a X b]) \oplus 1.
\]
We conjecture that $b_a=5$ for $a\in\{1,4\}$ and $b_a=0$ otherwise.
\end{conjecture}

If Conjecture \ref{conj:preperiod} holds, then together with the finite set of base evaluations
(Table~\ref{tab:single_piece}) it determines the SG values of all single-piece configurations; in
particular it determines which of them have value $0$ or $1$, yielding Conjecture \ref{conj:cover} for
$b$ in the periodic range.

This asymptotic regularity deteriorates for configurations with multiple interior pieces: states like
$[X a X b X]$ (both boundaries occupied) show no simple periodicity within computationally feasible
ranges.

In attempting to resolve these interior states, we have observed potential higher-order decompositions
and boundary absorptions, in which a short prefix or suffix of the board is absorbed into a fixed
value. For example, computational data supports the following structural simplifications for arbitrary
segments $u$ and $v$:
\begin{align*}
    \mathcal{G}([u X (5i+4) X v]) &= \mathcal{G}([u X 2]) \oplus \mathcal{G}([2 X v]) \oplus (i \bmod 2), \\
    \mathcal{G}([10 u]) &= \mathcal{G}([5 u]) \oplus 1, \\
    \mathcal{G}([8 X u]) &= \mathcal{G}([2 X u]), \\
    \mathcal{G}([X 7 X u]) &= \mathcal{G}([2 X u]) \oplus 1.
\end{align*}
We have proved the base case $i=0$ of the central splitting rule (this is Lemma~\ref{redrule3}(3));
proving the general cases and finding a unified algebraic framework for these prefix reductions remains
open.

\subsection{Generalization to Inverse $k$-cross}

A natural extension is Inverse $k$-cross, where placing a piece is prohibited if it completes a run of
$k$ pieces. Table \ref{tab:k_cross} lists SG values of the empty board of length $l$ for $k\le 13$.

\begin{table}[htpb]
    \centering
    \caption{SG values of the empty board for Inverse $k$-cross.}
    \label{tab:k_cross}
    \footnotesize
    \begin{tabular}{|c|cccccc|cccccc|}
        \hline
        $l \setminus k$ & 2 & 4 & 6 & 8 & 10 & 12 & 3 & 5 & 7 & 9 & 11 & 13 \\
        \hline
        0  & 0 & 0 & 0 & 0 & 0 & 0 & 0 & 0 & 0 & 0 & 0 & 0 \\
        1  & 1 & 1 & 1 & 1 & 1 & 1 & 1 & 1 & 1 & 1 & 1 & 1 \\
        2  & 1 & 0 & 0 & 0 & 0 & 0 & 0 & 0 & 0 & 0 & 0 & 0 \\
        3  & 2 & 1 & 1 & 1 & 1 & 1 & 0 & 1 & 1 & 1 & 1 & 1 \\
        4  & 0 & 1 & 0 & 0 & 0 & 0 & 1 & 0 & 0 & 0 & 0 & 0 \\
        5  & 3 & 0 & 1 & 1 & 1 & 1 & 1 & 0 & 1 & 1 & 1 & 1 \\
        6  & 1 & 2 & 1 & 0 & 0 & 0 & 0 & 1 & 0 & 0 & 0 & 0 \\
        7  & 1 & 1 & 0 & 1 & 1 & 1 & 1 & 0 & 0 & 1 & 1 & 1 \\
        8  & 0 & 0 & 1 & 1 & 0 & 0 & 1 & 0 & 1 & 0 & 0 & 0 \\
        9  & 3 & 1 & 2 & 0 & 1 & 1 & 0 & 1 & 0 & 0 & 1 & 1 \\
        10 & 3 & 2 & 0 & 1 & 1 & 0 & 0 & 0 & 1 & 1 & 0 & 0 \\
        11 & 2 & 0 & 1 & 0 & 0 & 1 & 1 & 1 & 1 & 0 & 0 & 1 \\
        12 & 2 & 1 & 0 & 2 & 1 & 1 & 0 & 0 & 0 & 1 & 1 & 0 \\
        13 & 4 & 1 & 1 & 1 & 0 & 0 & 0 & 0 & 1 & 0 & 0 & 0 \\
        14 & 0 & 0 & 0 & 0 & 1 & 1 & 1 & 1 & 0 & 0 & 1 & 1 \\
        15 & 5 & 1 & 3 & 1 & 2 & 0 & 1 & 0 & 1 & 1 & 0 & 0 \\
        16 & 2 & 0 & 1 & 0 & 0 & 1 & 0 & 0 & 0 & 0 & 1 & 1 \\
        17 & 2 & 0 & 0 & 1 & 1 & 0 & 1 & 1 & 1 & 1 & 1 & 0 \\
        18 & 3 & 1 & 1 & 0 & 0 & 2 & 1 & 0 & 1 & 0 & 0 & 1 \\
        19 & 3 & 2 & 1 & 1 & 1 & 1 & 0 & 1 & 0 & 1 & 1 & 0 \\
        20 & 0 & 0 & 0 & 3 & 0 & 0 & 0 & 0 & 1 & 0 & 0 & 0 \\
        21 & 1 & 1 & 1 & 0 & 1 & 1 & 1 & 0 & 0 & 1 & 1 & 1 \\
        22 & 1 & 0 & 0 & 1 & 0 & 0 & 0 & 1 & 0 & 0 & 0 & 0 \\
        23 & 3 & 3 & 1 & 0 & 1 & 1 & 0 & 0 & 1 & 0 & 1 & 1 \\
        24 & 0 & 1 & 0 & 1 & 0 & 0 & 1 & 0 & 0 & 1 & 0 & 0 \\
        25 & 2 & 2 & 0 & 1 & 3 & 1 & 1 & 1 & 1 & 0 & 1 & 1 \\
        \hline
    \end{tabular}
    \normalsize
\end{table}

This broader domain reveals a parity divergence:
\begin{enumerate}
    \item When $k=2$, the game is structurally isomorphic to Dawson's Kayles; Kenyon and Schwenk proved
    that its SG sequence is eventually periodic (period $34$ from $l=52$).
    \item When $k>2$ is even, the SG sequences appear chaotic, and it is unknown whether they are
    eventually periodic.
    \item When $k>2$ is odd, the empty board exhibits a symmetric periodicity mirroring our result for
    $k=3$.
\end{enumerate}

\begin{conjecture}\label{conj:k_odd}
For Inverse $k$-cross with $k=2m+1$ ($m\ge 1$), the SG value of the empty board $[a]$ satisfies
$\mathcal{G}([a])\le 1$. Furthermore,
\[
    \mathcal{G}([2+3m+a]) = \mathcal{G}([a]) \oplus (m \bmod 2),
    \qquad
    \mathcal{G}([2+5m-a]) = \mathcal{G}([a]) \oplus (m \bmod 2).
\]
\end{conjecture}

Our main theorem establishes this conjecture for $m=1$ ($k=3$). If proven globally, the recurrence
determines the SG values of the empty board for all odd $k>2$, since the base cases $a<k$ are given by
parity ($\mathcal{G}([a])=a\bmod 2$).

\subsection{Circular Chessboards}

Finally, Inverse Treblecross on a circular ring introduces nontrivial boundary conditions. Let $(l)$
denote an empty circular board of length $l$ (a cycle of $l$ cells). For $k=3$ the SG values are
partially periodic.

\begin{proposition}\label{prop:circle}
For $l\not\equiv 1,9\pmod{10}$, the SG value of the empty circular board is
\[
    \mathcal{G}((l)) =
    \begin{cases}
        0, & l \pmod{10} \in \{0,2,3,4,6,7,8\}; \\
        1, & l \pmod{10} = 5.
    \end{cases}
\]
\end{proposition}

\begin{proof}
For $l\pmod{10}\in\{0,2,3,4,6,7,8\}$ a first move places a piece, after which the circle becomes a
linear board. When $l\pmod{10}\in\{4,6,7,0,3\}$ the responding player plays adjacent to the first
piece, absorbing four cells and reducing to the linear board $[l-4]$, which by
Theorem~\ref{thm:main_SG} has $\chi_1(l-4)=1$, i.e.\ value $0$. When $l$ is even
($l\pmod{10}\in\{0,2,4,6,8\}$) the Tweedledum--Tweedledee strategy applies: the responding player
mirrors every move across the circle's axis of symmetry, which is possible because a cycle has no
``edge'' at which a mirror move could be blocked. For $l\equiv 5\pmod{10}$, a reduction analogous to
the linear analysis fractures the circle into a regular linear board $[v]\in R_1$ of value $0$.
\end{proof}

For $l\equiv 1\pmod{10}$ the sequence is not determined by the residue alone: $\mathcal{G}((11))=1$
while $\mathcal{G}((21))=0$. For $l\equiv 9\pmod{10}$ no simple periodicity is apparent either.

It is worth explaining why the Tweedledum--Tweedledee symmetry succeeds on circular boards of even
length but fails on linear ones. On a linear board a mirror move can be illegal: if the first player
plays near the centre to create $[w 1 X 4 w^R]$ (where $w^R$ denotes the left--right reversal of the
segment $w$), the mirror move must be separated by exactly two empty cells, forming
$[w 1 X 2 X 1 w^R]$; subsequent moves in this central zone frequently create three-in-a-row
infractions that break the symmetry.

\section{Acknowledgement}

We are grateful to Zhiyang Hang for performing the calculations for $n \le 50$. We also thank
Zi'ang Yan, Zhujun Zhang, and Zhaobo Han for computing the Sprague--Grundy values of the winning moves
and of the positions obtained by splitting a winning move, which revealed some suggestive patterns.
Finally, we thank Hongjie Zhang and Yunfei Lyu for their insightful comments.

\bibliographystyle{plain}
\bibliography{References}

\end{document}